\documentclass[10pt]{amsart}

\makeatletter
\def\blfootnote{\gdef\@thefnmark{}\@footnotetext}
\makeatother

\usepackage{epsfig}
\usepackage{graphics}
\usepackage{dcpic, pictexwd}

\usepackage[colorlinks=true,
                    linkcolor=blue,
                    urlcolor=blue,
                    citecolor=blue,
                    anchorcolor=blue]{hyperref}
\usepackage{mathtools}
\usepackage{amsmath,amssymb}
\newtheorem{theorem}{Theorem}[section]

\newtheorem{lemma}[theorem]{Lemma}
\newtheorem{proposition}[theorem]{Proposition}

\theoremstyle{remark}
\newtheorem{remark}[theorem]{Remark}
\theoremstyle{Acknowledgments}

\theoremstyle{definition}

\numberwithin{equation}{section}

 \def\Z{{\mathbb{Z}}}

\def\mod{{\rm Mod}}
\def\hmod{{\rm HMod}}

\begin{document}
 \blfootnote{\textup{2000} \textit{Mathematics Subject Classification}:
57M99, 20F38}
\newenvironment{prooff}{\medskip \par \noindent {\it Proof}\ }{\hfill
$\square$ \medskip \par}
    \def\sqr#1#2{{\vcenter{\hrule height.#2pt
        \hbox{\vrule width.#2pt height#1pt \kern#1pt
            \vrule width.#2pt}\hrule height.#2pt}}}
    \def\square{\mathchoice\sqr67\sqr67\sqr{2.1}6\sqr{1.5}6}
\def\pf#1{\medskip \par \noindent {\it #1.}\ }
\def\endpf{\hfill $\square$ \medskip \par}
\def\demo#1{\medskip \par \noindent {\it #1.}\ }
\def\enddemo{\medskip \par}
\def\qed{~\hfill$\square$}

 \title[The Number of Singular Fibers in Hyperelliptic Lefschetz Fibrations]
{The Number of Singular Fibers in Hyperelliptic Lefschetz Fibrations}

\author[T{\"{u}}l\.{i}n Altun{\"{o}}z]{T{\"{u}}l\.{i}n Altun{\"{o}}z}

 \address{Department of Mathematics, Middle East Technical University,
  Ankara, Turkey}
 \email{atulin@metu.edu.tr}


\begin{abstract}
We consider complex surfaces, viewed as smooth $4$-dimensional manifolds, that admit hyperelliptic Lefschetz fibrations over the $2$-sphere. In this paper, we show that the minimal number of singular fibers of such fibrations is equal to $2g+4$ for even $g\geq4$. For odd $g\geq7$, we show that the number is greater than or equal to $2g+6$.
Moreover, we discuss the minimal number of singular fibers in all hyperelliptic Lefschetz fibrations over the $2$-sphere as well.
\end{abstract}

 \maketitle
  \setcounter{secnumdepth}{2}
 \setcounter{section}{0}

\section{Introduction}
Donaldson and Gompf's results (\cite{Don.1}, \cite{Don.2}, \cite{Go} and \cite{GS}) give the relation between symplectic $4$-manifolds and Lefschetz fibrations, which are a fibering of a $4$-manifold by surfaces, with a finite number of singularities of a prescribed type. Donaldson proved that every symplectic $4$-manifold admits a Lefschetz pencil, which can be blown up at its base points to obtain a Lefschetz fibration. On the other hand, Gompf proved that any $4$-manifold admitting a Lefschetz fibration carries a symplectic
structure. The isomorphism class of a Lefschetz fibration is determined by its global monodromy. This relation provides a combinatorial way to understand any symplectic $4$-manifold via its monodromy, whenever it exists.\par
A Lefschetz fibration admits certain singular fibers associated to its monodromy. The results on the number of singular fibers of a Lefschetz fibration give us important information about the total space. It is well known that the number of singular fibers in a Lefschetz fibration cannot be arbitrary. A natural question to ask is what the minimal number of singular fibers in Lefschetz fibrations is. \par
Let $M_{g,h}$ denote the minimal number of singular fibers in all nontrivial relatively minimal Lefschetz fibrations of fiber genus $g$ and base genus $h$. Even though the exact value of $M_{g,h}$ for $h\geq1$ is almost known (except the numbers $M_{g,1}$ for $g\geq3$ and $M_{2,2}$) \cite{Ha,Ko,Ks,M,StY}, this question is still open when $h=0$ and $g\geq3$. It was proved that $M_{g,0}\leq 2g+4$ when $g$ is even and  $M_{g,0}\leq 2g+10$ when $g$ is odd~\cite{c,dmp,k1}. It is known that $M_{2,0}=7$ by Xiao's construction~\cite{X}. Recently, a relation among seven positive Dehn twists in the mapping class group of genus-$2$ surface was found by Baykur and Korkmaz~\cite{bk}. They also constructed an interesting relation consisting of $12$ positive Dehn twists along simple closed curves which are invariant under a hyperelliptic involution $\iota$ in the mapping class group of genus-$3$ surface. Moreover, they showed that the number of singular fibers in all genus-$3$ hyperelliptic Lefschetz fibrations over the $2$-sphere is greater than or equal to  $12$.\par
Let $N_{g}$ denote the minimal number of singular fibers in 
all genus-$g$  hyperelliptic Lefschetz fibrations over the $2$-sphere having at least one singular fiber. It follows from the result of Baykur and Korkmaz that $N_{3}=12$. For $g\geq4$, it is known that $N_{g}\leq2g+4$ (respectively $N_{g}\leq 8g+4$) when $g$ is even (respectively when $g$ is odd). (Here, $8g+4$ comes from the hyperelliptic relation.)\par
Let $M_{g}$ denote the minimal number of singular fibers in 
all genus-$g$  hyperelliptic Lefschetz fibrations on a complex surface over the $2$-sphere having at least one singular fiber. Here, by a complex surface we mean a compact connected complex analytic manifold of complex dimension $2$, considered as a smooth $4$-dimensional manifold.\par
Our aim in this paper is to estimate the numbers $N_{g}$ and $M_{g}$ for $g\geq4$. For the number $M_{g}$, we have the following results:
\begin{theorem}\label{thm1}
For all even $g\geq4$, $M_{g}=2g+4$.
\end{theorem}

\begin{theorem}\label{thm2}
For all odd $g\geq7$, $M_{g}\geq2g+6$.
\end{theorem}
For the number $N_{g}$ with $4\leq g \leq 10$, we have the following results:
\begin{theorem}\label{thm3}For the number $N_{g}$ the following holds.
\begin{enumerate}
\item[(1)] $N_{4}=12$, 
\item[(2)] $N_{5}\geq15$,
\item[(3)] $N_{6}=16$, 
\item[(4)] $N_{7}\geq17$,
\item[(5)] $N_{8}=19$ or $20$,
\item[(6)] $N_{9}\geq24$,
\item[(7)] $N_{10}=23$ or $24$.
\end{enumerate}

\end{theorem}

Here is an outline of the paper. In Section~\ref{S2}, we give some relevant background information from the theory of Lefschetz fibrations and some results to be used in the sequel. Section~\ref{S3} investigates the minimal number of singular fibers in  hyperelliptic Lefschetz fibrations on complex surfaces. In this section, we prove Theorems~\ref{thm1} and ~\ref{thm2}. In Section~\ref{S4}, we investigate the minimal number of singular fibers in hyperelliptic Lefschetz fibrations. We examine these numbers for $4\leq g \leq 10$ and prove Theorem~\ref{thm3}.
\medskip

\noindent
\textit{Acknowledgements.}
I would like to thank my advisor Mustafa Korkmaz for many invaluable comments and discussions. Thanks are due to Anar Akhmedov and T.-J. Li for helpful conversations. I also thank the referee and the editor for reading the paper very carefully, making many valuable suggestions and corrections. This paper is a part of the author's Ph.D. thesis ~\cite{a} 
 at Middle East Technical University. The author was partially supported by the Scientific and 
Technological Research Council of Turkey (T{\"{U}}B\.{I}TAK).


\section{ preliminaries}\label{S2}
\par  We start with a review of some basic definitions and properties of Lefschetz fibrations. In this paper, we denote the $2$-sphere by $\mathbb{S}^{2}$. Let $\Sigma_g$ denote a closed connected oriented surface of genus-$g$ and $\mod_g$ denote the mapping class group of $\Sigma_g$,  i.e., the group of isotopy classes of orientation-preserving diffeomorphisms of $\Sigma_g$. Let $M$ be a closed connected oriented smooth $4$-dimensional manifold. A smooth surjective map $f\colon M\to \mathbb{S}^{2}$ is a {\textit{Lefschetz fibration}} with connected oriented genus-$g$ regular fiber if it has finitely many critical points and around each critical point it is written in the form of  $f(z_{1},z_{2})=z_{1}^{2}+z_{2}^{2}$ with respect to some local complex coordinates agreeing with the orientations of $M$ and $\mathbb{S}^{2}$. The genus-$g$ of a regular fiber $F$ is called \textit{the genus of the fibration}. We assume that all the critical points lie in distinct fibers, called \textit{singular fibers}, which can be obtained after a small perturbation. Each singular fiber is obtained by shrinking a simple closed curve, called \textit{vanishing cycle}, in the regular fiber. If the vanishing cycle is non-separating (respectively separating), then the singular fiber is said to be \textit{irreducible} (respectively \textit{reducible}). In this paper, we also assume that all Lefschetz fibrations are nontrivial, i.e. it has at least one singular fiber and fibrations are relatively minimal, i.e. it has no fiber containing a sphere of self-intersection $-1$, otherwise one can blow-down it without changing the rest of the fibration.
\par
Lefschetz fibrations can be described combinatorially by their monodromy representations. The monodromy of a Lefschetz fibration $f\colon M\to \mathbb{S}^{2}$ is given by a positive factorization $t_{\alpha_{1}}t_{\alpha_{2}}\ldots t_{\alpha_{n}}=1$ in $\mod_g$, where $\alpha_{i}$ are the vanishing cycles of the singular fibers. (Here $t_a$ denotes the positive Dehn twist about a simple closed curve $a$ on a genus-$g$ surface.) Conversely, for a given  positive factorization $t_{a_{1}}t_{a_{2}}\ldots t_{a_{k}}=1$ in $\mod_g$, one can construct a genus-$g$ Lefschetz fibration over $\mathbb{S}^{2}$ by attaching $2$-handles along vanishing cycles $a_{i}$ in a $\Sigma_{g}$ fiber in $\Sigma_{g}\times D^{2}$ with $-1$ framing, and then by closing it up by a fiber preserving map to get a fibration over $\mathbb{S}^{2}$. Two Lefschetz fibrations $f_1\colon M_1\to \mathbb{S}^2$ and $f_2\colon M_2\to \mathbb{S}^2$ are said to be \textit{isomorphic} if there exist orientation preserving diffeomorphisms $H\colon M_1\to M_2$ and $h\colon \mathbb{S}^2 \to \mathbb{S}^2$ such that $f_2H=hf_1$.
If $g\geq2$, it is known that a genus-$g$ Lefschetz fibration over $\mathbb{S}^2$ is characterized by a positive factorization of the identity element  in $\mod_g$ up to \textit{Hurwitz moves} 
 (exchanging subwords $t_{a_{i}}t_{a_{i+1}}=t_{a_{i+1}}t_{t_{a_{i+1}}(a_{i})}$) and \textit{global conjugations} (changing each $t_{a_{i}}$ with $t_{\varphi(a_{i})}$ for some $\varphi \in \mod_g$).
\par The hyperelliptic mapping class group  $\hmod_g$  of $\Sigma_g$ is defined to be the subgroup of the mapping class group $\mod_g$ which is the centralizer of the class of a hyperelliptic involution $\iota\colon  \Sigma_g \to \Sigma_g$. We say that a genus-$g$ Lefschetz fibration is hyperelliptic if its vanishing cycles are invariant under the hyperelliptic involution $\iota$ up to isotopy.\par
We collect some useful facts about the first homology group of the hyperelliptic mapping class group.

\par Recall that for any group $G$, the first homology group of $G$ with integral coefficient is the abelianization of $G$, that is,
\[
H_1(G;\mathbb{Z})=G/[G,G],
 \]
  where $[G,G]$ is the subgroup of $G$ generated by all commutators $[a,b]=aba^{-1}b^{-1}$ for all $a,b \in G$. It is known that $H_1(\mod_g;\mathbb{Z})$ is a cyclic group generated by the class of a Dehn twist about a non-separating simple closed curve and also we have the following lemma:
  \begin{lemma}\label{hom} For a closed orientable surface of genus $g\geq1$, we have the following isomorphism of the first homology group $H_1(\mod_{g;}\mathbb{Z})$ of the mapping class group $\mod_{g}$:
  \[ H_1(\mod_{g};\mathbb{Z})\cong
  \begin{cases}
   \mathbb{Z}_{12},&\text{if $g=1$},\\
   \mathbb{Z}_{10},&\text{if $g=2$},\\
   0,&\text{if $g\geq3$}.\\
    \end{cases}
\]
\end{lemma}
For the proof of Lemma ~\ref{hom} and further details about the homology groups of the mapping class group, see \cite{k2}. \par
The following lemma can be proven by the presentation of the hyperelliptic mapping class group
 ~\cite{bh}.
\begin{lemma}\label{homh} For a closed orientable surface of genus $g\geq1$, the first
 homology group $H_1(\hmod_g;\mathbb{Z})$ of the hyperelliptic mapping class
  group $\hmod_{g}$ has the following isomorphism:
\[ H_1(\hmod_{g};\mathbb{Z})\cong
  \begin{cases}
   \mathbb{Z}/ 4(2g+1),&\text{if $g$ is odd},\\
   \mathbb{Z}/ 2(2g+1),&\text{if $g$ is even}.\\
    \end{cases}
\]
\end{lemma}
All Dehn twists about non-separating  simple closed curves that are invariant under the  hyperelliptic involution $\iota$ on $\Sigma_g$ are nontrivial in the hyperelliptic mapping class group of $\Sigma_g$, $\hmod_{g}$, and each of them maps to the same generator in $H_1(\hmod_{g};\Z)$ under the natural map $\hmod_{g}\to H_1(\hmod_{g};\Z)$. If a product of positive Dehn twists about non-separating curves in $\hmod_{g}$ is trivial then the number of twists is divisible by $4(2g+1)$ (respectively $2(2g+1$)) when $g$ is odd (respectively even). A separating simple closed curve on $\Sigma_g$ is said to be of \textit{type $h$} if it bounds subsurfaces of genera $h$ and $g-h$. By the even chain relation, each positive Dehn twist about a separating simple closed curve of type $h$ can be written as a product of $2h(4h+2)$ positive Dehn twists about non-separating simple closed curves. This implies the following relation between the number of non-separating singular fibers and that of separating singular fibers in a genus-$g$ hyperelliptic Lefschetz fibration:
\begin{lemma}\label{lem1}
Let $n$ (or $s$) be the number of non-separating (resp. separating) vanishing cycles in a genus-$g$ hyperelliptic Lefschetz fibration over $\mathbb{S}^{2}$. Then, we have
\begin{eqnarray}\label{eq1} 
n+\sum_{h=1}^{[g/2]}2h(4h+2)s_{h}\equiv  \begin{cases} 0 \pmod{4(2g+1)},&\text{if $g$ is odd,} \\ 
0 \pmod{2(2g+1)},&\text{if $g$ is even.}  \end{cases}
\nonumber 
\end{eqnarray}
where $s=\sum_{h=1}^{[g/2]}s_{h}$, and $s_h$ is the number of separating vanishing cycles of type $h$.
\end{lemma}

\begin{lemma}\cite{E,M1,M2}
Let $f\colon X \to \mathbb{S}^{2}$ be a genus-$g$ hyperelliptic Lefschetz fibration. 
Let $n$ and $s=\sum_{h=1}^{[g/2]}s_{h}$ be the numbers of non-separating
 and separating vanishing cycles of this fibration, respectively, where $s_h$ 
 denotes the number of separating vanishing cycles that separate the
genus-$g$ surface into two surfaces one of which has genus $h$. Then the 
signature of $X$ is
\begin{eqnarray*}
\sigma(X)&=&-\dfrac{g+1}{2g+1}n+\sum_{h=1}^{[g/2]} \bigg(\dfrac{4h(g-h)}{2g+1}-1\bigg)s_{h}.
\end{eqnarray*}
\end{lemma}

\begin{remark}
 Ozbagci~\cite{Oz.1} concluded that $\sigma(X)\leq n-s$ for any $4$-manifold $X$ admitting a genus-$g$ Lefschetz fibration over $\mathbb{S}^2$ or $\mathbb{D}^2$ and he also proved that 
\[
\sigma(X)\leq n-s-4
\]
 when the Lefschetz fibration over $\mathbb{S}^2$ is hyperelliptic. It can be easily obtained that $\sigma(X)\leq n-s-2$ using $b_1(X)\leq 2g-1$ by the handlebody decomposition of nontrivial Lefschetz fibrations over $\mathbb{S}^2$ and the fact that every nontrivial Lefschetz fibration over $\mathbb{S}^{2}$ has at least one non-separating vanishing cycle.
 Then, Cadavid~\cite{c} improved the upper bound of signature $\sigma(X)$, showing that
 \begin{eqnarray}\label{eq2} 
\sigma(X)\leq n-s-2(2g-b_1(X)).
\end{eqnarray}
\end{remark}
  \par
 Let us recall the following Stipsicz's theorem, which we will use to 
 examine the number of singular fibers.
 
\begin{theorem}\cite{St.1} \label{t1}
Let $f\colon X \to \mathbb{S}^2$ be a nontrivial genus-$g$ Lefschetz 
fibration with $b_{2}^{+}(X)=1$.
\begin{itemize}
\item[(1)]  If $g\geq 6$ is even, then $f\colon X \to \mathbb{S}^2$ admits at least $2g+4$ singular fibers. (This lower bound is sharp.)
\item [(2)] If $g\geq 15$ is odd, then $f\colon X \to \mathbb{S}^2$ 
admits at least $2g+10$ singular fibers. (This lower bound is sharp.)
\item[(3)]  If $g\geq 9$ is odd, then $f\colon X \to \mathbb{S}^2$ 
contains at least $2g+6$ singular fibers. 
\end{itemize}
\end{theorem}

We want to remark that in the above theorem the lower bounds in $(1)$ and $(2)$ are sharp, that is, the minimum values $2g+4$ and $2g+10$, respectively, can be realized on ruled surfaces which are uniquely determined as $(\Sigma _{g/2}\times \mathbb{S}^{2})\#4\overline{\mathbb{C} P^{2}}$ and $(\Sigma _{(g-1)/2}\times \mathbb{S}^{2})\#8\overline{\mathbb{C} P^{2}}$, respectively \cite[Sections $4.1$ and $4.2$]{St.1}. However, in $(3)$, the lower bound may not be sharp, that is, we do not know whether there exists a Lefschetz fibration $f\colon X \to \mathbb{S}^2$ with $b_{2}^{+}(X)=1$ and $2g+6$ singular fibers.
\section{The minimal number of singular fibers in hyperelliptic Lefschetz fibrations on complex surfaces}\label{S3}
\subsection{Even genus case}
In this section, first we prove some lemmas to prove Theorem ~\ref{thm1}.
\begin{lemma}\label{l1} The $4$-manifold 
$(\Sigma _{2}\times S^{2})\#3\overline{\mathbb{C} P^{2}}$
 does not admit a genus-$4$  Lefschetz fibration  over $\mathbb{S}^2$.
\end{lemma}
\begin{proof} 
Suppose that $(\Sigma _{2}\times \mathbb{S}^{2})\#3\overline{\mathbb{C} P^{2}}$ admits a genus-$4$ Lefschetz fibration and consider the homology class of a regular fiber $F$. We may write
\[
[F]=a[U]+b[V]+\displaystyle\sum_{i=1}^{3} c_{i}[E_{i}]\in H_{2}\big((\Sigma _{2}\times \mathbb{S}^{2})\#3\overline{\mathbb{C} P^{2}};\mathbb{Z}\big),
\]
for some integers $a$,$b$ and $c_{i}$,
where $[U]$, $[V]$ denote the homology classes of the section and fiber of the ruling $\Sigma _{2}\times S^{2}\to \Sigma _{2}$, respectively, such that $[U]^{2}=[V]^{2}=0$, $[U]\cdot [V]=1$, and $[E_{i}]$ denote the homology class of the exceptional sphere of the $i$th blow-up.

The composition of the blowing down and the projection map $\Sigma _{2}\times \mathbb{S}^{2}\to \Sigma _{2} $  leads to a degree-$d$ map $F\rightarrow \Sigma _{2}$ for some integer $d$. The degree $d$ must be equal to $a$. Moreover, since the fiber of the trivial $\mathbb{S}^{2}$-bundle $\Sigma _{2}\times \mathbb{S}^{2}\to \Sigma _{2}$ has a pseudo-holomorphic representative \cite{Liliu}, the degree of the map $F\rightarrow \Sigma _{2}$ is positive by the positivity of intersection.

Consider a singular fiber $\Sigma$. Since the normalization of $\Sigma$ has genus $\leq 3$, such a degree-$d$ map yields the following inequality
\[
3-1\geq g(\Sigma)-1 \geq d(2-1)=a(2-1),
\]
where $g(\Sigma)$ is the genus of the fiber $\Sigma$ ~\cite{Kne.1}.
Therefore, $0<d=a\leq 2$. Since $[F]^{2}=0$, we have 
\begin{eqnarray}
2ab=\displaystyle\sum_{i=1}^{3} c_{i}^{2}.\label{4eqn}
\end{eqnarray}
Since the symplectic structure on $(\Sigma _{2}\times \mathbb{S}^{2})\#3\overline{\mathbb{C} P^{2}} $ is unique up to deformations and diffeomorphisms, we can apply the adjunction formula
\[
2g(F)-2=[F]^{2}+[K]\cdot [F], 
\]
where $[K]=-2[U]+(2h-2)[V]+[E_{1}]+[E_{2}]+[E_{3}]$ is the canonical class with $h=g(\Sigma_2)=2$. In this case, the adjunction formula gives
\begin{eqnarray}\label{adjunction}
2g(F)-2=2ah-2a-2b-\displaystyle\sum_{i=1}^{3} c_{i}.
\end{eqnarray}
Thus, for $g(F)=4$ and $h=2$ , we have
\begin{eqnarray}
6=2a-2b-\displaystyle\sum_{i=1}^{3} c_{i}\label{adj1}.
\end{eqnarray}
For $a=1$, by the identities (\ref{4eqn}) and (\ref{adj1}) we have 
 \[
 \displaystyle\sum_{i=1}^{3} c_{i}^{2}=2b \mbox{ and } \displaystyle\sum_{i=1}^{3} c_{i} =-4-2b,
 \]
   which lead to
  \[
  \displaystyle\sum_{i=1}^{3} c_{i}^{2}+\displaystyle\sum_{i=1}^{3} c_{i} =-4.
  \]
  Hence $\displaystyle\sum_{i=1}^{3} \bigg(c_{i}+\frac{1}{2}\bigg)^{2}=-\frac{13}{4}$, which is not possible.

 In the case $a=2$, using the identities (\ref{4eqn}) and (\ref{adj1}), we have the following equalities:
 \[
 4b=\displaystyle\sum_{i=1}^{3} c_{i}^{2} \mbox{ and }
   2=-2-\displaystyle\sum_{i=1}^{3} c_{i},
   \]
  which give
  \[
  \displaystyle\sum_{i=1}^{3} c_{i}^{2}+2\displaystyle\sum_{i=1}^{3} c_{i}=-4.
  \]
 Thus, the resulting equality is $\displaystyle\sum_{i=1}^{3} (c_{i}+1)^{2}=-1$, which is a contradiction. Therefore, this shows that  $(\Sigma _{2}\times \mathbb{S}^{2})\#3\overline{\mathbb{C} P^{2}}$ does not admit a genus-$4$ Lefschetz fibration over $\mathbb{S}^2$.
 \end{proof}
\begin{remark}\label{rem1} The proof of the above theorem is based on Stipsicz's  
technique in \cite[Lemma $4.4$]{St.1}, and some arguments of \cite[Theorem $21$]{BK.1} 
and \cite[Lemma $4.2$]{Li.1}. It implies that Theorem \ref{t1} $(1)$ is true for $g=4$ and similarly, one can also show that Theorem \ref{t1} $(3)$ holds for $g=7$.
\end{remark}
Let  $e(X)$ denote the Euler characteristic of a $4$-manifold $X$.
For a genus-$g$ Lefschetz fibration $f\colon M \to \mathbb{S}^2$ with $n$ separating and $s$ non-separating vanishing cycles, we have 
\begin{eqnarray*}
 e(M)=4-4g+n+s.
 \end{eqnarray*}
We define the following two invariants associated to the $4$-manifold $M$:
\begin{eqnarray*}
 \chi_{h}(M)=\dfrac{e(M)+\sigma(M)}{4} \textrm{ and } c_{1}^{2}(M)=2e(X)+3\sigma(X).
 \end{eqnarray*}
Note that if $M$ is a complex surface, then $\chi_{h}(M)$ is the holomorphic Euler characteristic of $M$ and $c_{1}^{2}(M)$ is the square of the first Chern class of $M$.
\begin{lemma}\label{hollemma}
Let $f$ be a genus-$g$ hyperelliptic Lefschetz fibration on a complex surface $X$ over $\mathbb{S}^2$ with even $g\geq 6$ or odd $g\geq 9$. If $n+s<2g+4$, then $n\geq2g+2$.
\end{lemma}
\begin{proof}
Suppose that there exists a hyperelliptic Lefschetz fibration on a complex surface $X$ with $n< 2g+2$.

Let us first consider $n<2g$. Using the inequality $\sigma(X)\leq n-s-4$ for hyperelliptic Lefschetz fibrations over $\mathbb{S}^2$ ~\cite{Oz.1}, we have 
\begin{eqnarray*}
\chi_{h}(X)&=&\dfrac{e(X)+\sigma(X)}{4} \\
&\leq& \dfrac{(4-4g+n+s)+(n-s-4)}{4}\\
&=&\dfrac{2n-4g}{4}<0.
\end{eqnarray*}

Now, assume that $n=2g$, which gives rise to $s\leq3$. By the signature formula, we get
\begin{eqnarray*}
\sigma(X)&=&-\dfrac{g+1}{2g+1}n+\sum_{h=1}^{[g/2]} \Big(\dfrac{4h(g-h)}{2g+1}-1\Big)s_{h}\\
&\leq &-\dfrac{g+1}{2g+1}(2g)+3\Big(\dfrac{4(g/2)(g/2)}{2g+1}-1\Big)\\ 
&=&\dfrac{g^2-8g-3}{2g+1}\\
&<&\dfrac{g}{2}-3 
\end{eqnarray*}
and also, using $n+s\leq2g+3$ we have
\begin{eqnarray*}
\chi_{h}(X)&=&\dfrac{e(X)+\sigma(X)}{4}\\
&<&\dfrac{4-4g+2g+3+(g/2)-3}{4}\\
&\leq &\dfrac{-3(g/2)+4}{4}<0.
\end{eqnarray*}

Hence, we conclude that $\chi_{h}(X)<0$ if $n\leq2g$. By the classification of complex surfaces, $X$ is diffeomorphic to a blow up of a ruled surface which implies that $b_{2}^{+}=1$ ~\cite{ba}. However, this contradicts Theorem~\ref{t1}. Therefore, $n>2g$. Since the number $n$ is even by the equality in Lemma~\ref{lem1}, we get the required inequality.

\end{proof}
\subsubsection{Proof of Theorem~\ref{thm1}}

Suppose that  we have a hyperelliptic Lefschetz fibration on a complex surface $X$ with  $n+s<2g+4$ and $g\geq6$ is even. Hence, $n\geq 2g+2$ by Lemma~\ref{hollemma}. The equality in Lemma~\ref{lem1} implies that $n$ is even and also $s=\sum_{h=1}^{[g/2]}s_{h}>0$ . Thus, $n=2g+2$ and $s=1$.

The signature $\sigma(X)$ of $X$ is computed using the signature formula as follows:
\begin{eqnarray*}
\sigma(X)&=&-\dfrac{g+1}{2g+1}n+\sum_{h=1}^{[g/2]} \Big(\dfrac{4h(g-h)}{2g+1}-1\Big)s_{h}\\
&\leq &-\dfrac{g+1}{2g+1}(2g+2)+\Big(\dfrac{4(g/2)(g/2)}{2g+1}-1\Big)\\ 
&=&-\dfrac{g^2+6g+3}{2g+1}\\
&<&-\dfrac{g}{2}.
\end{eqnarray*}
Using $\sigma(X)<-\dfrac{g}{2}$,  $n=2g+2$ and $s=1$, we get:
\begin{eqnarray*}
\chi_{h}(X)&=&\dfrac{e(X)+\sigma(X)}{4}\\
&<&\dfrac{4-4g+2g+3-(g/2)}{4}\\
&\leq & \dfrac{-5(g/2)+7}{4}<0.
\end{eqnarray*}
In this case, the classification of complex surfaces implies that $X$ is a blow-up of a ruled surface and hence $b_{2}^{+}=1$. However, this is impossible if $g\geq6$ by Theorem~\ref{t1}. Thus $M_{g}\geq2g+4$. For even $g\geq6$, the existence of the genus-$g$ hyperelliptic Lefschetz fibration over $\mathbb{S}^2$ with $2g+4$ singular fibers ~\cite{c,dmp,k1} implies that $M_{g}=2g+4$.
\par
Now, consider the remaining case, $g=4$. Assume that there exists a hyperelliptic Lefschetz fibration so that $n+s<2g+4=12$. The equation in Lemma~\ref{lem1} leads to $n+12s_1+4s_2\equiv 0 \pmod{18}$, where $s=s_1+s_2$ and $n$ is even. 
Moreover, we have $n\geq6$ using the inequality $n\geq \dfrac{1}{5}(8g-3)=\dfrac{29}{5}$ given in~\cite{BK.1}.\par
The possible triples $(n,s_1,s_2)$ and some topological invariants of the corresponding
 genus-$4$ Lefschetz fibrations over $\mathbb{S}^2$, which can be easily computed 
 using the signature formula and $e(X)=4-4g+n+s=-12+n+s_1+s_2$, are given as follows:

\begin{center}
\begin{tabular}{ r|c|c|c|c| }
\multicolumn{1}{r}{}
 &  \multicolumn{1}{c}{$(n,s_1,s_2)$}
 & \multicolumn{1}{c}{$e(X)$} 
 & \multicolumn{1}{c}{$\sigma(X)$} 
 & \multicolumn{1}{c}{$c_{1}^{2}(X)$}  \\
\cline{2-5}
(a$1$) & $(6,1,0)$ & $-5$&$-3$&$-19$ \\
\cline{2-5}
(a$2$) & $(6,4,0)$  & $-2$ &$-2$&$-10$\\
\cline{2-5}
(a$3$) & $(6,0,3)$  & $-3$ &$-1$&$-9$\\
\cline{2-5}
(a$4$) & $(8,2,1)$  &$ -1$ &$-3$&$-11$\\
\cline{2-5}
\end{tabular}
\end{center}
\vspace*{0.4cm}
We now rule out all cases:

Case (a$1$). In this case, $c_1^2(X)=-19 < 4-4g=-12$, which gives a contradiction~\cite{St.2}. 

Cases (a$2$)$-$(a$4$). In these cases, $c_{1}^{2}(X)<2-2g=-6$. This implies that 
$X$ is a blow-up of a rational or ruled surface~\cite{Li.1}. Thus we have $b_{2}^{+}(X)=1$.  Moreover, 
using inequality (\ref{eq2}), one can conclude that
$X$ cannot be simply-connected and so it is a blow-up of a ruled surface.
 The equalities 
\begin{eqnarray*}
e(X)&=&2-2b_1(X)+b_{2}^{+}(X)+b_{2}^{-}(X)\\
&=&3-2b_1(X)+b_{2}^{-}(X)
\end{eqnarray*}
and 
\begin{eqnarray*}
\sigma(X)&=&b_{2}^{+}(X)-b_{2}^{-}(X)=1-b_{2}^{-}(X)
\end{eqnarray*}
imply that $b_{1}(X)=4$. Hence, $X$  is diffeomorphic to $(\Sigma _{2}\times S^{2})\#m\overline{\mathbb{C} P^{2}}$. (Note that $m=2, 1$ and $3$ for the cases (a$2$), (a$3$) and (a$4$), respectively). From the proof of Lemma~\ref{l1}, we see that  $(\Sigma _{2}\times S^{2})\#m\overline{\mathbb{C} P^{2}}$ cannot admit a genus-$4$  Lefschetz fibration over $\mathbb{S}^{2}$ for $m=0,1,2,3$. Since there is a hyperelliptic genus-$4$ Lefschetz fibration with $12$ singular fibers~\cite{c,dmp,k1}, we have $M_{4}=12$. This proves our claim.

\subsection{Odd genus case}
In this section, we find a lower bound for the number $M_{g}$ when $g\geq7$ is odd.
\subsubsection{Proof of Theorem~\ref{thm2}}
Suppose that there exists a hyperelliptic Lefschetz fibration on a complex surface $X$ with odd $g\geq 7$ and $n+s<2g+6$.

First consider the case $g\geq 9$. If $n<2g$, then it can be shown that $\chi_h(X)<0$ using the inequality $\sigma(X) \leq n-s-4$ as in the proof of Lemma~\ref{hollemma}. This implies that $b_{2}^{+}=1$ by the classification of complex surfaces. But, this gives a contradiction with Theorem~\ref{t1}. The odd case of the equation in Lemma~\ref{lem1} leads to $n\equiv 0 \pmod{4}$. We can conclude that $n\geq 2g+2$. The assumption $n+s<2g+6$ gives rise to $n=2g+2$ and $s\leq3$. Therefore, the signature formula implies the following inequality:
\begin{eqnarray*}
\sigma(X)&=&-\dfrac{g+1}{2g+1}n+\sum_{h=1}^{[g/2]} \Big(\dfrac{4h(g-h)}{2g+1}-1\Big)s_{h}\\
&\leq & -\dfrac{g+1}{2g+1}(2g+2)+3\Big(\dfrac{4(g/2)(g/2)}{2g+1}-1\Big)\\ 
&=&\dfrac{g^2-10g-5}{2g+1}\\
&<&\dfrac{g}{2}-5.
\end{eqnarray*}
Then, using the inequality $\sigma(X)<\dfrac{g}{2}-5$, the holomorphic Euler characteristic satisfies
\begin{eqnarray*}
\chi_{h}(X)&=&\dfrac{e(X)+\sigma(X)}{4}=\dfrac{4-4g+n+s+\sigma(X)}{4}\\
&<&\dfrac{4-4g+2g+5+(g/2)-5}{4}\\
&\leq&\dfrac{-3g}{8}+1<0.
\end{eqnarray*}
Hence, the classification of complex surfaces implies that $X$ is a blow-up of a ruled surface. In this case, $b_{2}^{+}(X)=1$. However, this contradicts to Theorem~\ref{t1}.

Now consider the case $g=7$. Suppose that we have a hyperelliptic  genus-$7$ Lefschetz fibration $X$ with $n+s<20$, where $s=s_1+s_2+s_3$. We know that 
$n\geq\dfrac{1}{5}(8g-3)=\dfrac{53}{5}$ (and therefore $n\geq11$) ~\cite{BK.1} and it follows from the congruence
\[
n+12s_1-20s_2+24s_3 \equiv 0\pmod{60}
\]
 that $n \equiv 0 \pmod{4}$  by Lemma~\ref{lem1}. Hence the possible values of  
 $(n,s_1,s_2,s_3)$, $e(X)$, $\sigma(X)$, $c_{1}^{2}(X)$ and $\chi_h(X)$ are as follows:
\begin{center}
\begin{tabular}{ r|c|c|c|c|c| }
\multicolumn{1}{r}{}
 &  \multicolumn{1}{c}{$(n,s_1,s_2,s_3)$}
 & \multicolumn{1}{c}{$e(X)$} 
 & \multicolumn{1}{c}{$\sigma(X)$} 
 & \multicolumn{1}{c}{$c_{1}^{2}(X)$}
 & \multicolumn{1}{c}{$\chi_h(X)$}  \\

\cline{2-6}
(b$1$) & $(12,0,0,2)$  &$ -10$ &$-2$&$-26$&$-3$\\
\cline{2-6}
(b$2$) & $(12,2,0,1)$  &$ -9$ &$-3$&$-27$&$-3$\\
\cline{2-6}
(b$3$) & $(12,4,0,0)$  &$ -8$ &$-4$&$-28$&$-3$\\
\cline{2-6}
(b$4$) & $(12,1,0,4)$  & $-7$ &$3$&$-5$&$-1$\\
\cline{2-6}
(b$5$) &$(12,0,3,2)$  & $-7$ &$3$&$-5$&$-1$\\
\cline{2-6}
(b$6$) & $(12,3,0,3)$  & $-6$ &$2$&$-12$&$-1$\\
\cline{2-6}
(b$7$) & $(12,2,3,1)$  &$ -5$ &$1$&$-7$&$-1$\\
\cline{2-6}
(b$8$) & $(12,0,0,7)$  &$ -5$ &$9$&$-5$&$1$\\
\cline{2-6}
(b$9$) & $(12,5,0,2)$  &$ -5$ &$1$&$-7$&$-1$\\
\cline{2-6}
(b$10$) &$ (12,4,3,0)$  &$ -5$ &$1$&$-7$&$-1$\\
\cline{2-6}
(b$11$) & $(16,0,2,1)$  &$ -5$ &$-3$&$-19$&$-2$\\
\cline{2-6}
\end{tabular}
\end{center}
\vspace*{0.4cm}

Cases (b$1$)$-$(b$3$). The manifold $X$ has $c_{1}^{2}(X)<4-4g=-24$, which gives a 
contradiction ~\cite{St.2}.

Cases (b$4$)$-$(b$7$), (b$9$) and (b$10$). In these cases, $\chi_h(X)<0$. Thus, $X$
is a blow-up of a ruled surface. However, $\sigma(X)\leq 0$ for such a manifold. Hence, we exclude these cases.

Case (b$8$). In this case, the manifold $X$ does not satisfy the inequality 
$\sigma(X)\leq n-s-4$.

Case (b$11$). In this case, since $c_{1}^{2}(X)<2-2g=-12$, $X$ is diffeomorphic to a blow-up of a rational or ruled surface. Hence $b_2^{+}=1$. We have 

\begin{eqnarray*}
e(X)&=&-5=2-2b_{1}(X)+b_{2}^{+}(X)+b_{2}^{-}(X)\\
&=&3-2b_1(X)+b_{2}^{-}(X)
\end{eqnarray*}
and 
\begin{eqnarray*}
\sigma(X)&=&-3=b_{2}^{+}(X)-b_{2}^{-}(X)\\
&=&1-b_{2}^{-}(X).
\end{eqnarray*}
Hence $(b_{1}(X),b_{2}^{+}(X),b_{2}^{-}(X))=(6,1,4)$. Therefore, 
$X=(\Sigma _{3}\times S^{2})\#3\overline{\mathbb{C}P^{2}}$. But one can prove that $(\Sigma _{3}\times S^{2})\#3\overline{\mathbb{C}P^{2}}$ does not admit a genus-$7$ Lefschetz fibration over $\mathbb{S}^2$ using the same idea as in the proof of Lemma~\ref{l1}. This finishes the proof.
\qed
\section{The minimal number of singular fibers in hyperelliptic Lefschetz fibrations} \label{S4}
In this section, we determine the minimal number of singular fibers in some hyperelliptic Lefschetz fibrations over $\mathbb{S}^{2}$. The proofs of Theorems~\ref{thm1} and~\ref{thm2} rely on the fact that any complex surface admitting a symplectic structure with $\chi_{h}<0$ is diffeomorphic to a ruled surface. In this section, we study the 
minimal number of singular fibers in hyperelliptic Lefschetz fibrations over $\mathbb{S}^{2}$ 
that may not have a complex structure. Recall that $N_{g}$ denotes
the minimal number of singular fibers in all hyperelliptic genus-$g$ Lefschetz fibrations over $\mathbb{S}^2$.

\subsubsection{Proof of Theorem~\ref{thm3}}
One can easily conclude that $N_{4}=12$ and $N_{7}\geq 17$ by the proofs of
Theorems~\ref{thm1} and ~\ref{thm2} (for $g=7$, refer to the cases (b$1$)$-$(b$3$)), respectively.

Now let us begin the proof of Theorem~\ref{thm3} (2).
Suppose that $N_{5}<15$ so that we have a hyperelliptic genus-$5$ Lefschetz fibration $X$. Let $n$ and $s=s_1+s_2$ be the numbers of nonseparating and separating vanishing cycles, respectively. Hence $n+s<15$.

The equation in Lemma~\ref{lem1} turns out to be
\[
n+12s_1-4s_2 \equiv 0 \pmod{44}
\]
so that $n$ is divided by $4$. It is known that $n\geq8$ ~\cite{BK.1}. The signature and the Euler characteristic are computed as
\[
  \sigma(X)=\dfrac{-6n+5s_1+13s_2}{11}
  \]
  and 
  \[
  e(X)=4-4g+n+s=-16+n+s_1+s_2,
  \]
  respectively. Hence the possible values of  $(n,s_1,s_2)$, $e(X)$, $\sigma(X)$, $c_{1}^{2}(X)$ and $\chi_h(X)$ are as follows:
\begin{center}
\begin{tabular}{ r|c|c|c|c|c| }
\multicolumn{1}{r}{}
 &  \multicolumn{1}{c}{$(n,s_1,s_2)$}
 & \multicolumn{1}{c}{$e(X)$} 
 & \multicolumn{1}{c}{$\sigma(X)$} 
 & \multicolumn{1}{c}{$c_{1}^{2}(X)$}
 & \multicolumn{1}{c}{$\chi_h(X)$}  \\
\cline{2-6}
(c$1$) & $(8,0,2)$ &$ -6$&$-2$&$-18$&$-2$ \\
\cline{2-6}
(c$2$) & $(8,3,0)$  &$ -5$ &$-3$&$-19$&$-2$\\
\cline{2-6}
(c$3$) & $(8,1,5)$  & $-2$ &$2$&$2$&$0$\\
\cline{2-6}
\end{tabular}
\end{center}
\vspace*{0.4cm}
We now eliminate all cases:

Cases (c$1$) and (c$2$). In these cases, $c_{1}^{2}(X)< 4-4g=-16$. This is impossible~\cite{St.2}.

Case (c$3$). In this case, $\sigma(X)> n-s-4$, which is also impossible 
for hyperelliptic Lefschetz fibrations~\cite{Oz.1}. Therefore, $N_{5}$ cannot be less than $15$.

Next, we will prove that $N_{6}=16$. Suppose that $N_{6}<16$ so that we have a hyperelliptic genus-$6$ Lefschetz fibration $X$ with $n+s<16$, where $s=s_1+s_2+s_3$. Using arguments similar to the above, we have 
the possible values of $(n,s_1,s_2,s_3)$, $e(X)$, $\sigma(X)$ and $ c_{1}^{2}(X)$ are as follows:
  \begin{center}
\begin{tabular}{ r|c|c|c|c| }
\multicolumn{1}{r}{}
 &  \multicolumn{1}{c}{$(n,s_1,s_2,s_3)$}
 & \multicolumn{1}{c}{$e(X)$} 
 & \multicolumn{1}{c}{$\sigma(X)$} 
 & \multicolumn{1}{c}{$c_{1}^{2}(X)$}  \\
\cline{2-5}
(d$1$) & $(10,0,3,0)$ & $-7$&$-1$&$-17$\\
\cline{2-5}
(d$2$) & $(10,3,0,1)$  &$ -6$ &$-2$&$-18$\\
\cline{2-5}
(d$3$) & $(10,2,0,3)$  & $-5$ &$1$&$-7$\\
\cline{2-5}
(d$4$) & $(10,1,4,0)$  & $-5$ &$1$&$-7$\\
\cline{2-5}
(d$5$) & $(12,0,1,0)$  &$ -7$ &$-5$&$-29$\\
\cline{2-5}
(d$6$) & $(12,1,2,0)$  & $-5$ &$-3$&$-19$\\
\cline{2-5}
(d$7$) &$ (14,1,0,0)$  & $-5$ &$-7$&$-31$\\
\cline{2-5}
\end{tabular}
\end{center}
\vspace*{0.4cm}
We now eliminate all cases:

Cases (d$5$) and (d$7$). In these cases, $c_1^2(X) < 4-4g=-22$. This is a contradiction~\cite{St.2}.

Cases (d$1$), (d$2$) and (d$6$). In these cases, $c_{1}^{2}<2-2g=-10$. Hence, $X$ is a blow-up of a rational or ruled surface~\cite{Li.1}. Thus, $b_{2}^{+}(X)=1$. However, this contradicts to Theorem~\ref{t1}.

Cases (d$3$) and (d$4$). In these cases, we have the following identities:
\begin{eqnarray}
  \sigma(X)&=&b_{2}^{+}(X)-b_{2}^{-}(X)=1, \label{heq5}\\
  e(X)&=&2-2b_1(X)+b_{2}^{+}(X)+b_{2}^{-}(X)=-5.\label{heq6}
  \end{eqnarray}

 So, the equations ($\ref{heq5}$) and ($\ref{heq6}$) yield 
\begin{eqnarray}
  b_{2}^{+}(X)&=&b_{1}(X)-3, \label{heq7}\\
   b_{2}^{-}(X)&=&b_{1}(X)-4.\label{heq8}
  \end{eqnarray}
Observe that $X$ cannot be a rational surface because $b_{1}(X)=4>0$ as $b_2^{+}=1$. Also, $X$ is not a blow-up of a ruled surface, since ruled surfaces have $\sigma\leq 0$. Let $\widetilde{X}$ be the minimal model of $X$ so that $X\cong\widetilde{X}\#k\overline{\mathbb{CP}}^{2}$ for some non-negative integer $k$. Due to Liu~\cite{Liu} and Taubes~\cite{Tau},  $c_{1}^{2}(\widetilde{X})\geq0$. Also, the equation 
\begin{eqnarray*}
c_{1}^{2}(\widetilde{X})=c_{1}^{2}(X)+k=-7+k
  \end{eqnarray*}
 implies that $k\geq7$. It is known that $b_{2}^{-}(X)\geq k \geq7$. The identity (\ref{heq8}) gives rise to $b_{1}(X)\geq 11$. Since $b_{1}(X)\leq 2g-1=11$  by the theory of Lefcshetz fibrations, we have $b_{1}(X)=11$. However, this contradicts with the result of \cite[Lemma 2.5]{Li.1}. Hence $N_{6}$ cannot be less than $16$. Since there exists a genus-$6$ hyperelliptic Lefschetz fibration with $16$ singular fibers~\cite{c,dmp,k1}, we have $N_{6}=16$.\par
For $8\leq g \leq10$, we list all possible values of the numbers $n$ and $s$ (the remaining details for these cases follow similarly from the above arguments). One can list these numbers using the congruence in Lemma~\ref{lem1} and the inequality  $n\geq(8g-3)/5$ ~\cite{BK.1}.\par
For $g=8$, the possible values of $(n,s_1,s_2,s_3,s_4), e(X)$, $\sigma(X)$ and $c_{1}^{2}(X)$ are as follows:

\begin{center}
\begin{tabular}{ r|c|c|c|c| }
\multicolumn{1}{r}{}
 &  \multicolumn{1}{c}{$(n,s_1,s_2,s_3,s_4)$}
 & \multicolumn{1}{c}{$e(X)$} 
 & \multicolumn{1}{c}{$\sigma(X)$} 
 & \multicolumn{1}{c}{$c_{1}^{2}(X)$}
 \\
\cline{2-5}
(e$1$) & $(14,1,0,0,1)$ & $-12$&$-4$&$-36$ \\
\cline{2-5}
(e$2$) & $(14,0,2,0,1)$  &$ -11$ &$-1$&$-25$\\
\cline{2-5}
(e$3$) & $(14,0,1,3,0)$  & $-10$ &$2$&$-14$\\
\cline{2-5}
(e$4$) & $(16,1,1,0,0)$  &$ -10$ &$-6$&$-38$\\
\cline{2-5}
(e$5$) & $(14,0,1,2,2)$  & $-9$ &$5$&$-3$\\
\cline{2-5}
\end{tabular}
\end{center}
\vspace*{0.4cm}
Using arguments similar to those arguments, we can eliminate all possibilities except for the case (e$5$). Thus, we can conclude that $N_{8}=19$ or $20$.\par       
For $g=9$, the possible values of $(n,s_1,s_2,s_3,s_4), e(X), \sigma(X)$ and $c_{1}^{2}(X)$ are as follows:
 \begin{center}
\begin{tabular}{ r|c|c|c|c| }
\multicolumn{1}{r}{}
 &  \multicolumn{1}{c}{$(n,s_1,s_2,s_3,s_4)$}
 & \multicolumn{1}{c}{$e(X)$} 
 & \multicolumn{1}{c}{$\sigma(X)$} 
 & \multicolumn{1}{c}{$c_{1}^{2}(X)$}
 \\
\cline{2-5}
(f$1$) & $(16,0,0,0,2)$ &$ -14$&$-2$&$-34$ \\
\cline{2-5}
(f$2$) & $(16,1,1,1,0)$  & $-13$ &$-3$&$-35$\\
\cline{2-5}
(f$3$) & $(16,0,0,1,3)$  & $-12$ &$4$&$-12$\\
 \cline{2-5}
(f$4$) & $(16,5,0,0,0)$  & $-11$&$-5$&$-37$\\
\cline{2-5}
(f$5$) & $(16,1,1,2,1)$  & $-11$ &$3$&$-13$\\
 \cline{2-5}
(f$6$) & $(16,0,3,2,0)$  & $-11$ &$3$&$-13$\\
\cline{2-5}
(f$7$) & $(16,3,1,0,2)$  & $-10$ &$2$&$-14$\\
\cline{2-5}
\end{tabular}
\end{center} 
\begin{center}
\begin{tabular}{ r|c|c|c|c| }
\multicolumn{1}{r}{}
 &  \multicolumn{1}{c}{$(n,s_1,s_2,s_3,s_4)$}
 & \multicolumn{1}{c}{$e(X)$} 
 & \multicolumn{1}{c}{$\sigma(X)$} 
 & \multicolumn{1}{c}{$c_{1}^{2}(X)$}
 \\
\cline{2-5}
(f$8$) & $(16,3,0,3,0)$  & $-10$ &$2$&$-14$\\
\cline{2-5}
(f$9$) & $(16,2,3,0,1)$  & $-10$ &$2$&$-14$\\
\cline{2-5}
(f$10$) & $(16,1,5,0,0)$  & $-10$ &$2$&$-14$\\
\cline{2-5}
(f$11$) & $(16,0,0,2,4)$  &$ -10$ &$10$&$10$\\
\cline{2-5}
(f$12$) & $(16,5,0,1,1)$  &$ -9$ &$1$&$-15$\\
\cline{2-5}
(f$13$) &$ (16,4,2,1,0)$  &$ -9$ &$1$&$-15$\\
\cline{2-5}
(f$14$) & $(16,2,0,0,5)$  &$ -9$ &$9$&$9$\\
\cline{2-5}
(f$15$) & $(16,1,2,0,4)$  &$ -9$ &$9$&$9$\\
\cline{2-5}
(f$16$) &$ (16,1,1,3,2)$  &$ -9$ &$9$&$9$\\
\cline{2-5}
(f$17$) & $(16,1,0,6,0)$  & $-9$ &$9$&$9$\\
\cline{2-5}
(f$18$) & $(16,0,4,0,3)$  & $-9$ &$9$&$9$\\
\cline{2-5}
(f$19$) & $(16,0,3,3,1)$  & $-9$ &$9$&$9$\\
\cline{2-5}
(f$20$) & $(20,0,1,2,0)$  & $-9$ &$-3$&$-27$\\
\cline{2-5}
\end{tabular}
\end{center} 
\vspace*{0.4cm}
Using arguments similar to those above, we can eliminate all possibilities (f$1$)$-$(f$20$) such that $n+s<24$. Thus, we can conclude that $N_{9}\geq 24$.\par  
For $g=10$, the possible values of $(n,s_1,s_2,s_3,s_4,s_5)$, $e(X)$, $\sigma(X)$ and $c_{1}^{2}(X)$ are as follows: 
	\begin{center}
\begin{tabular}{ r|c|c|c|c| }
\multicolumn{1}{r}{}
 &  \multicolumn{1}{c}{$(n,s_1,s_2,s_3,s_4,s_5)$}
 & \multicolumn{1}{c}{$e(X)$} 
 & \multicolumn{1}{c}{$\sigma(X)$} 
 & \multicolumn{1}{c}{$c_{1}^{2}(X)$}
  \\
\cline{2-5}
(g$1$) & $(16,0,1,0,1,1)$ & $-17$&$1$&$-31$ \\
\cline{2-5}
(g$2$) & $(16,1,2,0,1,0)$  &$ -16$ &$0$&$-32$\\
\cline{2-5}
(g$3$) & $(16,0,1,1,1,1)$  &$ -16$ &$4$&$-20$\\
\cline{2-5}
(g$4$) & $(16,1,2,1,1,0)$  & $-15$&$3$&$-21$\\
\cline{2-5}
(g$5$) & $(16,1,0,0,2,2)$  &$ -15$ &$7$&$-9$\\
\cline{2-5}
(g$6$) & $(16,0,2,0,0,3)$  & $-15$ &$7$&$-9$\\
\cline{2-5}
(g$7$) & $(16,0,1,2,1,1)$  &$ -15$ &$7$&$-9$\\
\cline{2-5}
(g$8$) & $(16,4,0,0,0,2)$  & $-14$ &$2$&$-22$\\
\cline{2-5}
(g$9$) & $(16,2,1,0,2,1)$  & $-14$ &$6$&$-10$\\
\cline{2-5}
(g$10$) & $(16,1,3,0,0,2)$  &$ -14$ &$6$&$-10$\\
\cline{2-5}
(g$11$) & $(16,1,2,2,1,0)$  &$ -14$ &$6$&$-10$\\
\cline{2-5}
(g$12$) & $(16,1,0,1,2,2)$  & $-14$ &$10$&$2$\\
\cline{2-5}
(g$13$) & $(16,0,2,1,0,3)$  & $-14$ &$10$&$2$\\
\cline{2-5}
(g$14$) & $(16,0,2,0,4,0)$  & $-14$ &$10$&$2$\\
\cline{2-5}
(g$15$) & $(16,0,0,0,1,5)$  &$ -14$ &$14$&$14$\\
\cline{2-5}
(g$16$) & $(18,2,0,0,0,0)$  & $-16$ &$-8$&$-56$\\
\cline{2-5}
(g$17$) & $(18,2,0,1,0,0)$  &$ -15$ &$-5$&$-45$\\
\cline{2-5}
(g$18$) & $(18,2,0,2,0,0)$  &$ -14$ &$-2$&$-34$\\
\cline{2-5}
(g$19$) & $(18,1,0,0,3,0)$  & $-14$ &$2$&$-22$\\
\cline{2-5}
(g$20$) & $(18,0,2,0,0,1)$  & $-14$ &$2$&$-22$\\
\cline{2-5}
(g$21$) & $(20,1,0,0,0,1)$  & $-14$ &$-6$&$-46$\\
\cline{2-5}
(g$22$) &$ (18,0,0,0,2,3)$  & $-13$ &$9$&$1$\\
\cline{2-5}
\end{tabular}
\end{center}
\vspace*{0.4cm}
Using arguments similar to those above arguments, one can eliminate all possibilities except for the case (g$22$). Thus, one can conclude that $N_{10}=23$ or $24$.      
\qed 
\vspace*{2.5mm}
 \par As long as genus-$g$ increases, the number of possibilities of $n$ and $s$ increases, where $n$ and $s$ are the numbers of irreducible and reducible fibers, respectively. Hence, it is hard to find the exact value of $N_{g}$. The odd case is harder because of the upper bound $8g+4$ of $N_{g}$. For the general case we have the following:
\begin{proposition}\label{pr}
Let $f\colon X\to \mathbb{S}^2$ be a genus-$g$ Lefschetz fibration with $n+s<2g+4$ and $g>6$. Then the signature of $X$, $\sigma(X)$, is positive.
\end{proposition}
\begin{proof}
Suppose that $X$ admits a genus-$g$ Lefschetz fibration with $n+s<2g+4$ for $g>6$ and $g\neq 7$. It follows from Theorem~\ref{t1} that $b_2^{+}\neq1$ and therefore $X$ is not a blow-up of a rational or ruled surface. This gives $c_1^{2}(X)\geq 2-2g$ by ~\cite{Li.1}. Therefore we get:
\begin{eqnarray*}
2-2g \leq c_1^{2}(X)&=&3\sigma(X)+2e(X)\\
&=&3\sigma(X)+2(4-4g+n+s)\\
&\leq & 3\sigma(X)+2(4-4g+2g+3)\\
&=&3\sigma(X)+14-4g,
\end{eqnarray*}
 which implies that $\sigma(X)>0$ when $g> 6$ and $g\neq7$.\par
 We see that the same argument holds for $g=7$ using Remark~\ref{rem1}.
 \end{proof} 

\begin{remark}
 Proposition~\ref{pr} implies that every hyperelliptic genus-$g$ 
Lefschetz fibration with $n+s<2g+4$ and $g>6$ has 
$b_1(X)>\dfrac{8g-15}{6}$ by applying the inequality (\ref{eq2}). However, 
the existence of such a Lefschetz fibration is not known~\cite{b1, Sm}.
\end{remark}
\begin{remark}
 Recently, Korkmaz has constructed a factorization of the identity in the
hyperelliptic mapping class group $\hmod_{g}$ with length $5g-3$.  
This new construction provides us to improve the upper bound of $N_{g}$ 
when $g$ is odd. Therefore, we conclude that $N_{g}\leq 2g+4$ if $g$ is even and $N_{g}\leq 5g-3$ if $g$ is odd.
\end{remark}



\begin{thebibliography}{xxxx}
\bibitem{a}T. Altun{\"{o}}z: \emph{ Exotic $4$-manifolds and hyperelliptic Lefschetz fibrations}, Ph.D. thesis, Middle
East Technical University, Ankara, Turkey, 2018.

\bibitem{ba}W. Barth, K. Hulek, C. Peters and A. Van de Ven:   Compact Complex Surfaces, 2nd edn. Springer, Heidelberg, 2004.

\bibitem{b1}R. I. Baykur: \emph{ Small symplectic Calabi-Yau surfaces and exotic $4$-manifolds via genus-$3$ pencils},
preprint, arXiv:1511.05951.
\bibitem{bk} R. I. Baykur and M. Korkmaz:
\emph{ Small Lefschetz fibrations and exotic $4$-manifolds},
Math. Ann. {\bf{367}} (3-4), (2017), 1333--1361.


\bibitem{bh} J. S. Birman and H. Hilden:  \emph{ On the mapping class groups of closed surfaces as covering spaces}, Advances in the theory of Riemann surfaces. Ann. Math. Stud. {\bf 66} (1971), 81--115.

\bibitem{BK.1} V. Braungardt, D. Kotschick:  \emph{ Clustering of critical points in Lefschetz fibrations and the symplectic Szpiro Inequality}, Trans. Amer. Math. Soc. {\bf 355} (8) (2003), 3217--3226.

 \bibitem{c} C. Cadavid:
\emph{ A remarkable set of words in the mapping class group},
 Ph.D dissertation, Univ. of Texas, Austin, 1998.

 \bibitem{dmp} E. Dalyan, E. Medeto\^{g}ullari, M. Pamuk:
\emph{ A note on the generalized Matsumoto relation},
 Turk. J. Math. {\bf 41}(2017), 524--536.

\bibitem{Don.1} S. K. Donaldson:  \emph{ Lefschetz fibrations in symplectic geometry}, Doc. Math. J. DMV.
Extra Volume, ICMII (1998), 309--314.

\bibitem{Don.2} S. K. Donaldson: \emph{ Lefschetz pencils on symplectic manifolds}, J. Differential Geom. {\bf 53} (2) (1999), 205--236. 
\bibitem{E} H. Endo: \emph{ Meyer’s signature cocycle and hyperelliptic fibrations}, Math. Ann. {\bf 316} (2000), 237--257.



 
\bibitem{Go} R. E. Gompf:\emph{ The topology of symplectic manifolds}, Turkish J. Math. {\bf 25} (2001), 43--59. 
\bibitem{GS} R. E. Gompf and A. I. Stipsicz: $4$-manifolds and Kirby calculus, Graduate Studies in Mathematics, vol. 20, American Math. Society, Providence 1999.

\bibitem{Ha} N. Hamada: \emph{ Upper bounds for the minimal number of singular fibers in a Lefschetz fibration
over the torus}, Michigan Math. J. {\bf 63} (2) (2014), 275--291. 


\bibitem{Kne.1} H. Kneser: \emph{ Die kleinste Bedeckungszahl innerhalb einer Klasse von Fl\"a chenabbildungen}, Math. Ann. {\bf 103} (1930), 347--358. 


 \bibitem{k1} M. Korkmaz:
\emph{ Noncomplex smooth $4$-manifolds with Lefschetz fibrations},
 Internat. Math. Res. Notices {\bf 2001} (2001) (3), 115--128.
 
 \bibitem{k2} M. Korkmaz:
\emph{ Low-dimensional homology groups of mapping class groups: a survey}, Turk. J. Math. {\bf 26} (1) (2002), 101--114.
 
\bibitem{Ko} M. Korkmaz and B. Ozbagci: \emph{ Minimal number of singular fibers in a Lefschetz fibration}, Proc.
Amer. Math. Soc. {\bf 129} (5) (2001), 1545--1549. 



\bibitem{Ks} M. Korkmaz and A. Stipsicz: \emph{ Lefschetz fibrations on $4$-manifolds}, Handbook of Teichm\"{u}ller
theory. Vol. II, IRMA Lect. Math. Theor. Phys., vol. 13, Eur. Math. Soc., Z¨{u}rich, 2009, pp. 271--296. 


\bibitem{Li.1}  T.-J. Li:\emph{ Symplectic Parshin-Arakelov inequality}, Internat.Math. Res. Notices {\bf 2000} (18) (2000), 941--954.

\bibitem{Liliu}  T.-J. Li and A. Liu: \emph{ Symplectic structure on ruled surfaces and a generalized adjunction formula}, Math. Res. Lett. {\bf 2} (4) (1995), 453--471.
\bibitem{Liu} A. Liu: \emph{ Some new applications of the general wall crossing formula}, Math. Res. Lett.
{\bf 3} (1996), 569--585.
 \bibitem{M1} Y. Matsumoto:\emph{ On $4$-manifolds fibered by tori II}, Proc. Japan Acad., {\bf 59A} (1983),100--103.
 
 \bibitem{M2} Y. Matsumoto:\emph{ Lefschetz fibrations of genus two — a topological approach}, Proceedings of
the $37$th Taniguchi Symposium on Topology and Teichm{\"{u}}ller Spaces, ed. Sadayoshi
Kojima et al., World Scientific (1996), 123--148, 
 
 \bibitem{M} N. Monden: \emph{ On minimal number of singular fibers in a genus-$2$ Lefschetz fibration}, Tokyo J.
Math. {\bf 35} (2) (2012), 483--490. 

 \bibitem{Oz.1} B. Ozbagci:\emph{ Signatures of Lefschetz fibrations}, Pacific J. Math. {\bf 202} (1) (2002), 99--118. 
 \bibitem{Sm} I. Smith: \emph{ Torus fibrations on symplectic four-manifolds}, Turk. J. Math. {\bf 25} (2001), 69--95.
\bibitem{St.2} A. Stipsicz:\emph{ On the number of vanishing cycles in Lefschetz fibrations}, 
Math. Res. Lett.
{\bf 6} (3-4) (1999), 449--456.
\bibitem{St.1} A. Stipsicz:\emph{ Singular fibers in Lefschetz fibrations on manifolds with $b_{2}^{+}= 1$}, Topology
Appl. {\bf 117} (1) (2002),  9--21.
\bibitem{StY} A. Stipsicz and K.-H. Yun:\emph{ On minimal number of singular fibers in Lefschetz fibrations over the torus}, Proc. Amer. Math. Soc. {\bf 145} (8) (2017), 3607--3616.
\bibitem{Tau} C. Taubes: \emph{SW =$\Longrightarrow Gr:$ From the 
Seiberg-Witten equations to pseudo-holomorphic curves}, J. Amer. Math. Soc. {\bf 9} (1996), 845--918.
\bibitem{X} G. Xiao: Surfaces fibr\'{e}e en courbes de genre deux, Lecture Notes in Mathematics, {\bf 1137} Springer-Verlag, Berlin, 1985.





\end{thebibliography}
\end{document}